\documentclass[preprint,12pt,authoryear]{elsarticle}

\usepackage[english]{babel}
\usepackage{amsmath}
\usepackage{amssymb}
\usepackage{amsthm}
\usepackage{bbm}
\usepackage{proba}
\usepackage{import}
\usepackage[utf8]{inputenc}
\usepackage{imakeidx}
\usepackage{natbib}
\usepackage{color}
\usepackage{setspace}
\usepackage{enumerate}
\usepackage{tikz}
\usepackage{tcolorbox}
\usepackage[citecolor = blue,
            anchorcolor = blue]{hyperref}
\usepackage{amsmath, amssymb}

\newcommand\norm[1]{\left\lVert#1\right\rVert}
\newcommand\inn[1]{\left\langle#1\right\rangle}

\newtheorem{theorem}{Theorem}[section]
\newtheorem{corollary}{Corollary}[theorem]

\newtheorem{proposition}{Proposition}[section]

\def\R{\mathbb{R}}
\def\H{\mathcal{H}}
\def\S{\mathbb{S}}
\def\d{{\rm d}}
\def\F{\mathcal{F}}
\def\P{\mathbb{P}}
\def\Y{\mathcal{Y}}
\def\L{\mathcal{L}}
\def\B{\mathcal{B}}

\begin{document}

\begin{frontmatter}



\title{Functional Gaussian Fields on Hyperspheres with their Equivalent Gaussian Measures} 


\author[1]{Alessia Caponera}
\author[2]{Vinicius Ferreira}
\author[3,4]{Emilio Porcu}

\address[1]{Luiss University, Department of AI, Data and Decision Sciences, Roma, Italy}
\address[2]{IMPA, Rio de Janeiro, Brasil}
\address[3]{Khalifa University, Department of Mathematics, Abu Dhabi, The United Arab Emirates}
\address[4]{ADIA Lab, Abu Dhabi, The United Arab Emirates}

\begin{abstract}
We develop a general framework for isotropic functional Gaussian fields on the $d$-dimensional sphere $\S^{d}$, where the field takes values in a separable Hilbert space $\H$. We establish an operator-valued extension of Schoenberg’s theorem and show that the covariance structure of such fields admits a representation in terms of a sequence of trace-class $d$-Schoenberg operators. This yields an explicit spectral decomposition of the covariance operator on $L^{2}(\S^{d};\H)$. \par
We then derive a functional version of the Feldman-Hájek criterion and prove that equivalence of the Gaussian measures induced by two Hilbert-valued spherical fields is determined by a Hilbert summability criterion that involves Schoenberg functional sequences,
thereby extending classical results for scalar and vector fields on spheres to the infinite-dimensional setting. We further show how equivalence of all scalar projections is contained within, and dominated by, the functional criterion.
\par 
The theory is illustrated through two classes of models: (i) a multiquadratic bivariate family on $\S^{d}$, for which the equivalence region can be expressed in closed form in terms of cross-correlation and geodesic decay parameters, and (ii) an infinite-dimensional Legendre-Matérn construction, where operator-valued spectra lead to explicit identifiability conditions on smoothness and scale parameters. These examples demonstrate how the operator-valued Schoenberg coefficients govern both the geometry and the measure-theoretic behavior of functional spherical fields.
\par
Overall, the results provide a unified spectral framework for Gaussian measures on $L^{2}(\S^{d};\H)$, bridging harmonic analysis, operator theory, and stochastic geometry on manifolds, and offering foundational tools for functional data analysis, spatial statistics, and kernel methods on spherical domains. 
\end{abstract}


\end{frontmatter}



\section{Introduction}

Functional data analysis (FDA) has become a central subject in many branches of theoretical and applied sciences. Within the paradigm of FDA, each datum is an element taking values over infinite dimensional spaces such as separable Hilbert spaces \citep{RamsaySilverman2005, FerratyVieu2006, HsingEubank2015}. Typical examples include curves or surfaces indexed by time, space or both, as well as images in remote sensing, trajectories in neuroscience and biomechanics, to mention just a few. Treating data as Hilbert-valued objects allows to connect FDA with solid theory coming from stochastic processes, functional and harmonic analysis within a unified framework.  \par
Stochastic processes defined over high dimensional spaces (space or space-time) are customarily called random fields or random functions \citep{Matheron:1965}. Functional Gaussian fields, which are random fields having an Hilbert space as an image set, have been largely used within both the statistics and machine learning communities \citep{BerlinetThomasAgnan2004, RasmussenWilliams2006}. In particular, the theory of reproducing kernel Hilbert spaces  has been especially popular within the realm of high dimensional fields, and perfectly suitable to treat observations that can be conceived as random objects defined over Hilbert spaces. Predictive modeling for functional random fields has attracted large interest in the last 30 years, and the reader is referred to \cite{giraldo2011ordinary, menafoglio2013universal, menafoglio2014kriging} and \cite{delicado2010statistics}, with the references therein.

\par  
Gaussian random fields defined over (hyper) spheres have a long history that traces back to harmonic analysis, probability theory, statistics and machine learning. Recently, they have become very popular as they are the backbone for the construction of statistical models for climate model outputs \citep{crippa2016population, porcu202130, CM21,caponera22EJS}. While the theory of Gaussian processes on the sphere is well developed for scalar and vector-valued (finite-dimensional) fields, little is known about the functional case. {A key reference in this direction is \cite{caponera}, which develops the spectral theory for the 2-dimensional sphere.} \par
For Gaussian fields, the finite dimensional distributions are completely specified through the covariance functions, which are positive definite. A natural assumption for covariance functions over spheres is that of geodesic isotropy \citep{porcu202130}, that is the covariance between any pair of random variables located over two different points depends exclusively on their geodesic distance, being the arccosine of the dot product between any two points located over the spherical shell. Characterization of scalar positive definite functions over $d$-dimensional spheres is due to \cite{schoenberg1942}, and subsequent generalizations to \cite{berg2017schoenberg} for the case of product spaces involving the sphere. Characterization of matrix-valued positive definite functions over spheres is due to \cite{Hannan:1980}, while the functional case is elusive so far. This theory has allowed for the construction of covariance models that are indispensable for phenomena intrinsically defined on spherical domains, including the Cosmic Microwave Background (CMB) in cosmology and geophysical variables (temperature, pressure, wind, ocean currents) in Earth system science. \par
Equivalence of Gaussian measures is of paramount importance in spatial statistics under the paradigm of infill asymptotics. They represent a tool {\em sine qua non} to evaluate asymptotic effects of misspecified best linear unbiased prediction (kriging), that is when prediction is performed with a wrong covariance function \citep[see][with the references therein, for a thorough treatment]{Stein:1999}. Further, they cover fundamental importance to evaluate the asymptotic performance of certain likelihood-based estimators for the parameters indexing any parametric class of covariance functions. While this subject was originally born within the realm of abstract probability theory \citep{skorokhod_yadrenko, Yadrenko1983, feldman}, in the last 30 years there has been a massive use of this theory to have new theoretical findings in both spatial statistics and machine learning communities. \par 
Equivalence of Gaussian measures over $d$-dimensional spheres has been studied by \cite{arafat2018equivalence} for scalar random fields. Equivalence of functional Gaussian fields has been studied over planar surfaces endowed with Euclidean metrics only, and the case of the sphere is once again elusive.  Extending these characterizations to Hilbert-valued fields requires an operator-valued analogue of the Schoenberg decomposition and a functional version of the Feldman-H\'ajek criterion. 

\medskip
The plan of the paper is the following. Section \ref{sec2} contains the necessary background material needed to understand the subsequent mathematical developments, that are provided in Section \ref{sec3}. This section contains novel theoretical results in concert with some examples. Section \ref{sec4} concludes the paper with a discussion.

\section{Background} \label{sec2}

\subsection{Scalar and Vector Gaussian Fields on Spheres}

Some parts of the paper will make use of Gaussian fields defined over finite dimensional spaces. \\
Let $Z$ be a Gaussian random field defined over the $d$-dimensional sphere with unit radius, $\S^d$, embedded in $\R^{d+1}$. We supposed $Z$ to have finite first and second order moments, so that the covariance function is well defined. Additionally, we shall work under the assumption of geodesic isotropy, so that the covariance $R:\S^d \times \S^d \to \R $ depends on the inner product between any two points located over the spherical shell. Harmonic analysis results coming from \cite{schoenberg1942} allow to decompose the function $R$ (whenever continuous) as a sum of the type 
\begin{equation}
    \label{schoenberg-1} R(x,y) = \sum_{l \ge 0} b_l C_{l}^{(d-1)/2} (x \cdot y), \qquad x,y \in \S^d, 
\end{equation}
with $C_{l}^{k}$ denoting the $l$th Gegenbauer polynomial of order $k$ \citep[see][with the references therein]{berg2017schoenberg} and where the coefficients are nonnegative with the sum $\sum_l b_l= \sigma^2$, with $\sigma^2$ denoting the variance of $Z$. We should have actually used the notation $b_{l,d}$ to emphasize that the coefficients depend on the dimension of the sphere. However, this fact is not relevant to the exposition following subsequently, and would be an additional source of obfuscation as the findings in subsequent sections are mathematically and notationally involved. \cite{berg2017schoenberg}, mimicking \cite{daley-porcu}, adopt the illustrative nomenclature $d$-{Schoenberg coefficients} for the scalars $b_l$ and consequently $\{b_l\}_{l \ge 0}$ is called a $d$-Schoenberg sequence. The extension for vector valued Gaussian fields $Z \in \R^p$, with $p$ a positive integers, can be found in \cite{Hannan:1980}, in this case, the covariance $R$ is matrix valued function from $\S^d \times \S^d $ into $\R^{p \times p}$, and the Schoenberg expansion (\ref{schoenberg-1}) is similarly attained: 
\begin{equation}
    \label{schoenberg-2} R(x,y) = \sum_{l \ge 0} B_l C_{l}^{(d-1)/2} (x \cdot y), \qquad x,y \in \S^d, 
\end{equation}
where $\{B_l\}_{l \ge 0}$ is a summable sequence of positive semidefinite matrices. \par
Let $(\Omega, {\cal A}, {\bf P}_i)$, $i=1,2$ be two probability spaces. Two measures ${\bf P}_i$ are called equivalent when they are mutually absolutely continuous. Otherwise, they are called orthogonal. Gaussian measures are the only ones having either equivalence or orthogonality (no indetermination cases). For Gaussian measures associated with two random fields, equivalence is intended {\em along the paths} of the random fields \citep{Zhang:2004}. Since Gaussian measures are uniquely determined by their mean and covariance, it is not a mystery that any condition regarding equivalence of Gaussian measures translates into conditions for the respective covariance functions, or equivalently their spectrum. \par
Equivalence of Gaussian measures associated with random fields defined over hyperspheres has been studied by \cite{arafat2018equivalence}, who proved that 
\begin{equation} \label{arafat-condition}
\displaystyle {\bf P}_1 \equiv {\bf P}_2 \iff \sum_{l \geq 0} h(l) \left( \frac{b_l^{(1)}}{b_l^{(2)}} - 1 \right)^2<\infty, \end{equation}  
where $\{b_i^{(i)}\}_{l \ge 0}$ are the $d$-Schoenberg sequences associated with the covariance functions $R_i$, $i=1,2$, and there $\{ h(l) \}_{l \ge 0}$ is a sequence of nonnegative numbers that are defined explicitly in the subsequent section (see also \cite{arafat2018equivalence}).
 
\subsection{Functional Gaussian Fields on Spheres}

\paragraph{Notation.}
Let us denote by $\H$ a separable Hilbert space, with its inner product $\inn{\cdot,\cdot}_{\H}$ and the induced norm $\norm{\cdot}_{\H}$. Denote by $x \otimes y$ the operator acting as $(x \otimes y) z = \inn{z,y}_{\H}x$. The most important case is that when $\norm{x}_{\H}=1$, the operator $x \otimes x$ is the orthogonal projection onto the direction of $x$. The space of Hilbert-Schmidt operators is denoted by $\L_2(\H)$, with its norm given by $\norm{A}_{\L_2(\H)}^2 = \sum_{k \geq 1} \norm{Ae_k}_{\H}^2$, where $\{e_k\}_{k \geq 1}$ is any orthonormal basis. The Euclidean dot product between two vectors $x,y \in \R^d$ will be denoted by $x \cdot y$ and the Euclidean norm $\norm{x}$. Let $\S^d \subset \R^{d+1}$ be the unit sphere. We denote $L^2(\S^d)$ the usual space of square-integrable scalar functions on the sphere with the inner product $\inn{f,g}_{L^2(\S^d)} := \int_{\S^d}f(x)g(x) \d \sigma(x)$, where $\d \sigma(\cdot)$ is the standard volume element on the sphere {(Haar measure)}. 
It is well known that $L^{2}(\mathbb{S}^{d})$ decomposes as a direct sum of mutually orthogonal subspaces spanned by the eigenfunctions of the Laplacian on $\mathbb{S}^{d}$. These eigenfunctions are referred to as hyperspherical harmonics. A standard orthonormal basis for the eigenspace of degree $l \ge 0$ is given by the (in this paper, \emph{real-valued}) \emph{fully normalized hyperspherical harmonics} $\{\mathcal{Y}_{l,m}, m = 1,\dots,h(l)\}$, with
$$
h(l) := \frac{(2l + d - 1)\,(l + d - 2)!}{l!\,(d - 1)!}.$$
We denote with $L^2(\S^d;\H)$  the space of square-integrable functions on the sphere taking values on $\H$, {\em i.e.,} $f:\S^d \to \H$ such that $\int_{\S^d}\norm{f(x)}_{\H}^2 \d \sigma(x) < \infty$. This is also a Hilbert Space with the inner product $\inn{f,g}_{L^2(\S^d;\H)} := \int_{\S^d} \inn{f(x),g(x)}_{\H} \d \sigma(x)$. The operators $f \otimes_{L^2(\S^d)} g$ and $f \otimes_{L^2(\S^d;\H)} g$ are defined similarly as before, for $f,g \in L^2(\S^d)$ or $L^2(\S^d;\H)$ respectively. Given a positive semi-definite operator $B$ in any of the Hilbert spaces described above, we define $B^{1/2}$ as the unique positive semi-definite operator satisfying $B^{1/2}B^{1/2} = B$. If additionally $B$ is injective, $B^{-1}$ is the (possibly unbounded) operator with domain $D(B^{-1}) := R(B)$ given by $B^{-1}\phi = f \iff Bf = \phi$. {For $(\Omega, \F, \P)$ a probability space, a $\H$-valued random variable $F$ is a measurable
map from $(\Omega, \F)$ onto $(\H, \mathfrak{B}(\H))$, $\mathfrak{B}(\H)$ denoting the Borel $\sigma$-field of $\H$. The map $F$  is Gaussian if, for all $ w$ in $\H$, the real-valued random variable $\langle F, w \rangle_\H$ is Gaussian.}\\


We consider a collection $\{Z(x), \ x \in \mathbb{S}^d\}$ of $\H$-valued random variables defined on a common probability space $(\Omega, \F, \P)$ and such that $\mathbb{E} \|Z(x)\|^2_\H < \infty$, for any $x \in \mathbb{S}^d$. We also assume that the mapping $Z : \Omega \times \S^d \to \H$ is \emph{jointly measurable}, i.e., measurable with respect to the product $\sigma$-field $\mathfrak{B}(\mathbb{S}^d) \times \F$, $\mathfrak{B}(\mathbb{S}^d)$ denoting the Borel $\sigma$-field of $\mathbb{S}^d$. We call $\{Z(x), \ x \in \mathbb{S}^d\}$ \emph{$\H$-valued spherical random field}. The mean element is well-defined as Bochner integral and for simplicity we set it to be the zero element of $\H$ (for more details, see \cite{HsingEubank2015}).

We say that the collection $\{Z(x), \ x \in \mathbb{S}^d\}$ of zero-mean $\H$-valued random variables is geodesically isotropic if $\mathbb{E} \|Z(x) \|_\H^2 < \infty$, for all $x \in \mathbb{S}^d$, and
for any $x,y \in \S^d$ the covariance function $R(x,y):= \E[Z(x) \otimes_{{\H}} Z(y)]$ 
depends only on $x\cdot y$. With some abuse of notation, we write $R(x,y) = R(x \cdot y)$. 

{
Under joint measurability with the additional assumption of isotropy, it is readily seen that $\mathbb{E} \left [ \int_{\mathbb{S}^d } \|T(x)\|^2_\H d\sigma(x)  \right] < \infty,$ and hence $Z(\omega, \cdot)$ is an element of $L^2(\mathbb{S}^d; \H)$ for $\P$-almost every $\omega \in \Omega$. In particular, the covariance operator $\B := \E[Z \otimes_{L^2(\S^d;\H)} Z]$ is well defined on $L^2(\S^d;\H)$. 
}

The spectral theory for such fields is well developed in \cite{caponera}, for the case $d=2$. The proofs for general dimension $d$ follow clearly the same steps and we state the result here without proof.

\begin{theorem}
    {Let $\{Z(x), \ x \in \mathbb{S}^d\}$ an isotropic $\H$-valued spherical random field.}
    Then the following decomposition holds:
    \begin{equation}\label{Z_decomposition}
        Z(x) = \sum_{l=0}^\infty\sum_{m=1}^{h(l)} a_{l,m}\Y_{l,m}(x), \qquad {x \in \S^d,}
    \end{equation}
    where the coefficients are $\H$-valued {random} variables given by the Bochner integral $a_{l,m}:=  \int_{\S^d}Z(x)\Y_{l,m}(x) \d\sigma(x)$. Additionally, for all $l,l' \geq 0$, $1\leq m \leq h(l)$ and $1 \leq m' \leq h(l')$, {it is true that} \begin{equation}\label{b_definition}
        \E [a_{l,m} \otimes_{\H} a_{l',m'}] = b_l \delta_{l,l'}\delta_{m,m'},
    \end{equation}
    where $b_l: \H \to \H$ is a collection of positive semi-definite trace class operators. Furthermore, the convergence in (\ref{Z_decomposition}) occurs in the following sense:
    \begin{eqnarray*}
        && \E  \norm{Z - \sum_{l=0}^L\sum_{m=1}^{h(l)}a_{l,m}\Y_{l,m}}_{L^2(\S^d;\H)}^2 \\&=&  \E \left[ \int_{S^d} \norm{Z(x) - \sum_{l=0}^L\sum_{m=1}^{h(l)}a_{l,m}\Y_{l,m}(x)}_{\H}^2 \d\sigma(x)  \right] \xrightarrow[L \to \infty]{} 0
    \end{eqnarray*}
    and 
    \begin{equation*}
       \sup_{x \in \mathbb{S}^d}   \, \E  \norm{Z(x) - \sum_{l=0}^L\sum_{m=1}^{h(l)}a_{l,m}\Y_{l,m}(x)}_{\H}^2 \xrightarrow[L \to \infty]{} 0.
    \end{equation*}
\end{theorem}
As an immediate consequence of representation (\ref{Z_decomposition}) and the orthogonality relations (\ref{b_definition}) we have the decomposition of the covariance function $R({\cdot,\cdot})$: 
\begin{equation}\label{R_harmonics}
    \begin{split}
        R(x,y) &= \E [Z(x)\otimes_{\H}Z(y)] 
        \\
        &= \sum_{\substack {l \geq 0\\ l' \geq 0}} \sum_{\substack {1\leq m \leq h(l)\\ 1\leq m' \leq h(l')}} \E[a_{l,m}\otimes_{\H}a_{l',m'}]\Y_{l,m}(x)\Y_{l,m}(y) 
        \\
        &= \sum_{l\geq 0} \sum_{m=1}^{h(l)} b_l \Y_{l,m}(x)\Y_{l,m}(y), \qquad {x,y \in \S^d.}
    \end{split}
\end{equation}
If we recall the summation identity for the spherical harmonics in terms of Gegenbauer  polynomials \citep{berg2017schoenberg}, $$C_l^{(d-1)/2}(x \cdot y) = \sum_{m=1}^{h(l)} \Y_{l,m}(x)\Y_{l,m}(y), \qquad x,y \in \S^d, $$ and then plug into the last identity in (\ref{R_harmonics}), the dependence of $R$ on $x \cdot y$ becomes {apparent:}
\begin{equation}\label{R_gegenbauer}
    R(x,y) = \sum_{l \geq 0} b_l C_l^{(d-1)/2}(x \cdot y),\qquad {x,y \in \S^d.}
\end{equation}
This identity reveals much more, as it is the natural extension of Schoenberg theorem \citep{schoenberg1942} to the functional case. Extending the notations from \cite{berg2017schoenberg} and \cite{daley-porcu}, it seems natural to call the sequence $\{b_l\}_{l \ge 0}$ a $d$-functional Schoenberg sequence whose elements map $\H$ onto itself and are additionally trace class. \par 
We now turn our attention into 
the covariance operator 
$\B = \E[Z \otimes_{L^2(\S^d;\H)} Z]$. 
Such operator acts on functions $f \in L^2(\S^d;\H)$ as
\begin{equation}
    \begin{split}
        (\B f)(x) = \E[\inn{f,Z}_{L^2(\S^d;\H)}Z](x) &= \E \left[ \int_{\S^d} \inn{f(y),Z(y)}_{\H} \d \sigma(y) Z(x) \right]
        \\ 
        &= \int_{\S^d} \E[\inn{f(y),Z(y)}_{\H}Z(x)] \d \sigma(y)
        \\
        &= \int_{\S^d} \E[(Z(x) \otimes_{\H} Z(y)) f(y)] \d \sigma(y)
        \\
        &= \int_{\S^d} R(x,y)f(y) \d \sigma(y), \qquad x \in \S^d.
    \end{split}
\end{equation}
The above chain of identities can then be used in concert with (\ref{R_harmonics}) to get
\begin{equation}\label{B_schoenberg}
\begin{split}
    (\B f)(x) &= \int_{\S^d} \sum_{l\geq 0}\sum_{m=1}^{h(l)} b_l \Y_{l,m}(x)\Y_{l,m}(y)f(y) \d \sigma(y)
    \\
    &= \sum_{l \geq 0}b_l \sum_{m=1}^{h(l)}\int_{\S^d} \Y_{l,m}(y)f(y) \d \sigma(y) \Y_{l,m}(x)
    \\
    &= \sum_{l\geq 0} b_l (P_lf) (x), \qquad x \in \S^d,
\end{split}
\end{equation} 
where $$ ( P_lf ) (x)= \sum_{m=1}^{h(l)}\inn{f,\Y_{l,m}}_{L^2(\S^d)}\Y_{l,m}(x), \qquad x \in \S^d.$$ If $f \in L^2(\S^d)$, $P_l$ is simply the orthogonal projection of $f$ onto ${\bf H}_l := \text{ span }\{ \Y_{l,m} \}_{m=1}^{h(l)}$. If $f \in L^2(\S^d;\H)$, which is the case in the last sum above, then $P_l$ acts as the same projection done on each component $f_k \in L^2(\S^d)$ of $f = (f_1,f_2,\ldots)$ when decomposed in coordinates of an orthonormal basis of $\H$. 


Hence, we have $P_lP_{l'} = P_l \delta_{l,l'}$, $P_{l'}\Y_{l,m} = \delta_{l,l'}\Y_{l,m}$ for $l\geq 0, m=1,\ldots,h(l)$ and:
\begin{equation}\label{l2norm}
    \norm{f}_{L^2(\S^d;\H)}^2 = \sum_{l {\ge 0 }} \norm{P_lf}_{L^2(\S^d;\H)}^2.
\end{equation}
Additionally, since $\Y_{l,m}$ are scalar functions, we have $b_lP_lf = P_lb_lf$.

\section{Results} \label{sec3}

\subsection{A Functional Feldman-Hajek Form}
Now consider two such random fields $Z^{(j)}$ with associated sequences of operators $\{ b_l^{(j)}\}_{l \geq 0}$, $j=1,2$, which we assume to be all strictly positive. We are interested in necessary and sufficient conditions for their respective distributions ${\bf P}_1$ and ${\bf P}_2$ to be equivalent (one being absolutely continuous with respect to the other) in terms of $\{b_l^{(j)}\}_{l\geq 0}$, $j=1,2$. 
The result following subsequently extends the condition (\ref{arafat-condition}) to the case where $Z$ is a functional Gaussian field as previously described. We anticipate that we provide a constructive proof that is a consequence of 
Feldman-Hájek Theorem \citep{feldman,Hajek1958}, which in turn requires that $\B_2^{-1/2}\B_1 \B_2^{-1/2} - I$ is a Hilbert-Schmidt operator on $L^2(\S^d;\H)$. We start by studying the representation for the operator $\B_2^{-1/2}$. In the scalar case ($\H = \R$), the coefficients $b_l \in \R$ are the eigenvalues associated with eigenbasis $\Y_{l,m}$ and for $\alpha \in \R$ we have $B^{\alpha}f = \sum_{l \geq 0} b_l^{\alpha}P_lf$, provided that the sum converges in the $L^2$ norm. Below we show ann analogous decomposition holds.

{
\begin{proposition}\label{B_inverse}
    Let $\B$ admit a decomposition as in (\ref{B_schoenberg}). Then, the domain of $\B^{-1/2}$ is given by: 
    $$D(\B^{-1/2}) := R(\B^{1/2}) = \left \{ u \in L^2(\S^d;\H) ; \sum_{l\geq 0} \norm{b_l^{-1/2} P_l u}_{L^2(\S^d;\H)}^2 < \infty \right \},$$
and $\B^{-1/2}$ admits the representation
\begin{equation}\label{B_inverse_schoenberg}
    \B^{-1/2}u = \sum_{l \geq 0} b_l^{-1/2}P_l u.
\end{equation}
\end{proposition}
\begin{proof}
    We {first} verify that $\B^{1/2} = \sum_{l \geq 0}b_l^{1/2} P_l$. Indeed, direct inspection proves
    \begin{equation*}
        \Big(\sum_{l\geq 0}b_l^{1/2} P_l\Big)^2 = \sum_{l,l' \geq 0}b_l^{1/2}b_{l'}^{1/2}P_lP_{l'}=\sum_{l,l' \geq 0}b_l^{1/2}b_{l'}^{1/2}P_l\delta_{l,l'} = \sum_{l \geq 0} b_l P_l =\B.
    \end{equation*}
To characterize $D(\B^{-1/2})$, we first take $u \in D(\B^{-1/2}) = R(\B^{1/2})$. Thus, there {exists} an {element} $f \in L^2(\S^d;\H)$ such that $$u = \B^{1/2}f = \sum_{l \geq 0}b_l^{1/2}P_lf = \sum_{l \geq 0}P_lb_l^{1/2}f.$$ Thus, for any $l \ge0$, $P_l u =  P_l b_l^{1/2} f$, and hence $b_l^{-1/2}P_l u = P_l f$, from which we obtain $$\sum_{l \geq 0} \norm{b_l^{-1/2} P_l u}_{L^2(\S^d;\H)}^2 = \sum_{l \geq 0} \norm{P_l f}_{L^2(\S^d;\H)}^2 = \norm{f}_{L^2(\S^d;\H)}^2 < +\infty.$$\\
Conversely, if $u \in L^2(\S^d;\H)$ is such that $\sum_{l\geq 0} \norm{b_l^{-1/2} P_l u}_{L^2(\S^d;\H)}^2 < +\infty$, then $f := \sum_{l\geq 0}b_l^{-1/2} P_l u$ is well defined, and using the representation for $\B^{1/2}$ above we get:
\begin{align*}
    \B^{1/2}f &= \sum_{l' \geq 0} b_{l'}^{1/2}P_{l'}f = \sum_{l,l' \geq 0}b_{l'}^{1/2}b_l^{-1/2}P_{l'}P_l u = \sum_{l,l' \geq 0}b_{l'}^{1/2}b_l^{-1/2}P_l \delta_{l,l'}u \\&= \sum_{l \geq 0}b_l^{1/2}b_l^{-1/2}P_l u = \sum_{l \geq 0}P_l u = u.
\end{align*}
Thus $u \in D(\B^{-1/2})$, with $\B^{-1/2}u = f$ satisfying (\ref{B_inverse_schoenberg}), which concludes the proof.
\end{proof}}
The characterization above for the operator $\B^{1/2}$ is the crux of the argument for providing equivalence conditions for Gaussian measures associated with $\H$-valued fields. 
{To state the following result, we introduce some additional notation. Throughout, $I$ is the identity operator on $\H$.}
\begin{theorem}
    For $j=1,2$, let $Z^{(j)}$ be an isotropic $\H$-valued Gaussian random field defined on $\S^d$, where its distribution ${\bf P}_j$ has covariance operator $\B_j$ with its associated sequence of operators $\{b_l^{(j)} \}_{l \geq 0}$ according to (\ref{B_schoenberg}). Then, ${\bf P}_1 \equiv {\bf P}_2$ if and only if:
\begin{equation}\label{equiv_condition}
        \sum_{l \geq 0}h(l) \norm{(b_l^{(2)})^{-1/2}b_l^{(1)}(b_l^{(2)})^{-1/2} - I}_{\L_2(\H)}^2 < +\infty.
    \end{equation}

\end{theorem}
\begin{proof}
    By combining the representations for $\B_1^{1/2}$ and $\B_2^{-1/2}$ as in Theorem \ref{B_inverse}, we obtain
    \begin{equation}
        D:=\B_2^{-1/2}\B_1 \B_2^{-1/2} - I = \sum_{l\geq 0} [(b_l^{(2)})^{-1/2}b_l^{(1)}(b_l^{(2)})^{-1/2} - I] P_l.
    \end{equation}
    By Feldman-Hájek Theorem, ${\bf P}_1 \equiv {\bf P}_2$ {if and only if}  $D$ is a Hilbert-Schmidt operator on $L^2(\S^d;\H)$. To conclude the proof we hence need to verify that the Hilbert-Schmidt norm of $D$ is given by the left hand side of (\ref{equiv_condition}). \par   Let $\psi_{k,l,m} := e_k \Y_{l,m} \in L^2(\S^d;\H)$, where $\{ e_k\}_{k \geq 0}$ is an orthonormal basis of $\H$. Since $\Y_{l,m}$ is an orthonormal basis of $L^2(\S^d)$, it follows that ${\{\psi_{k,l,m}\}} $ is an orthonormal basis of $L^2(\S^d;\H)$. Denoting $d_l := (b_l^{(2)})^{-1/2}b_l^{(1)}(b_l^{(2)})^{-1/2} - I$, we can write:
    \begin{equation}
        \begin{split}
            \norm{D}_{\L_2(L^2(\S^d;\H))}^2 &= \sum_{k,l,m} \norm{D\psi_{k,l,m}}_{L^2(\S^d;\H)}^2
            \\
            &= \sum_{k,l,m}\norm{\sum_{l'} (d_{l'} P_{l'}) (e_k \Y_{l,m})}_{ L^2(\S^d;\H)}^2
            \\
            &= \sum_{k,l,m}\norm{\sum_{l'} (d_{l'} e_k) \delta_{l,l'}\Y_{l,m}}_{L^2(\S^d;\H)}^2 \qquad (\text{ because } P_{l'} \Y_{l,m} = \delta_{l,l'}\Y_{l,m})
            \\
            &=\sum_{k,l,m} \norm{(d_l e_k)\Y_{l,m}}_{L^2(\S^d;\H)}^2
            \\
            &=\sum_{k,l,m}\int_{\S^d}\norm{(d_l e_k) \Y_{l,m}(x)}_{\H}^2\d \sigma(x)
            \\
            &=\sum_{l,m}\int_{\S^d}\Y_{l,m}(x)^2\d\sigma(x)\sum_{k}\norm{d_l e_k}_{\H}^2
            \\
            &=\sum_{l\geq 0}\sum_{m=1}^{h(l)}\norm{d_l}_{\L_2(\H)}^2 \quad (\text{using} \norm{\Y_{l,m}}_{L^2(\S^d)}=1 \text{ and the definition of } \norm{\cdot}_{\L_2(\H)})
            \\
            &=\sum_{l \geq 0} h(l)\norm{(b_l^{(2)})^{-1/2}b_l^{(1)}(b_l^{(2)})^{-1/2} - I}_{\L_2(\H)}^2.
        \end{split}
    \end{equation}
This concludes the proof.
\end{proof}

\subsection{Scalar Marginalization}

{Next, we discuss how to relate the problems of equivalence between the vector fields $Z_1$ and $Z_2$ with the equivalence of its scalar components $\inn{Z_1,u}_{\H}$ and $\inn{Z_2,u}_{\H}$, $u \in \H$. For that purpose 
$Z_u := \inn{Z,u}_{\H}$. Its covariance function $R_u$ at any two points $x,y \in \R^d$  is given by
\begin{equation*}
    \begin{split}
        R_u(x,y) = \E[Z_u(x)Z_u(y)] &= \E[\inn{Z(x),u}_{\H}\inn{u,Z(y)}_{\H}]
        \\
        &= \E[\inn{Z(x)\otimes_{\H}Z(y)u,u}_{\H}]
        \\
        &= \inn{R(x,y)u,u}_{\H}.
    \end{split}
\end{equation*}
The covariance operator $\B_u$ of $Z_u$ acts on scalar functions $f \in L^2(\S^d)$ as
\begin{equation*}
    \begin{split}
        ( \B_uf ) (x) = \int_{\S^d} R_u(x,y)f(y) \d \sigma(y) &= \int_{\S^d}\inn{R(x,y)u,u}_{\H}f(y) \d \sigma(y)
        \\
        &= \inn{\int_{\S^d}R(x,y)uf(y) \d \sigma(y), u}_{\H}
        \\
        &= \inn{\B(uf)(x),u}_{\H}.
    \end{split}
\end{equation*}
Using the series representation of $\B$ from (\ref{B_schoenberg}) we get
\begin{equation}\label{Bu_schoenberg}
    \B_uf = \inn{\sum_{l \geq 0}b_l P_l(uf),u}_{\H} = \sum_{l \geq 0} \inn{b_l u,u}_{\H}P_l f,
\end{equation}
which shows that the scalar Schoenberg coefficients of $\B_u$ are $\inn{b_lu,u}$, where $b_l$ are the operator valued Schoenberg coefficients of $\B$. 
Let us now consider a pair of such random fields $Z_u^{(j)}$, $j=1,2$ as defined above. By using (\ref{Bu_schoenberg}) we get that the operator $D_u := [(\B_u^{(2)})^{1-2} \B_u^{(1)}(\B_u^{(2)})^{1-2} - I]$ can be written as:
\begin{equation*}
    D_u = \sum_{l \geq 0} \left( \frac{\inn{b_l^{(2)}u,u}_{\H}}{\inn{b_l^{(1)}u,u}_{\H}} - 1 \right) P_l.
\end{equation*}
Thus, the necessary and sufficient condition for the distributions $\bf{P}_u^{(j)}$ of $Z_u^{(j)}$ to be equivalent is 
\begin{equation}\label{equiv_condition2}
    \sum_{l \geq 0} h(l) \left( \frac{\inn{b_l^{(2)}u,u}_{\H}}{\inn{b_l^{(1)}u,u}_{\H}} - 1 \right)^2 < +\infty.
\end{equation}
It is straightforward from the definition of equivalence that if $\bf{P}^{(1)} \equiv \bf{P}^{(2)}$ then $\bf{P}_u^{(1)} \equiv \bf{P}_u^{(1)}$ for all $u \in \H$. However, we prove below that the left hand side of (\ref{equiv_condition2}) is bounded by the one in \ref{equiv_condition}, which is useful in the next section when identifying parameters that ensure equivalence of measures in some classes of Gaussian fields. First we prove the following inequality.
\begin{proposition}\label{proposition_norms}
    Let $A,B$ be linear operators on $\H$ such that $B$ is a positive-definite, self-adjoint and compact. Then, for any $u \in \H$,
    \begin{equation}\label{inequality1}
        \left| \frac{\inn{(A-B)u,u}_{\H}}{\inn{Bu,u}_{\H}} \right| \leq \norm{B^{-1/2}AB^{-1/2} - I}_{\L_2(\H)}.
    \end{equation}
\end{proposition}
\begin{proof}
    Let $\{ e_n\}_{n \geq 1}$ be an orthonormal basis of eigenvectors of $B$, with $Be_n = \lambda_n e_n$. Let $u \in \H$. By writing $u = \sum_{n \geq 0} u_n e_n$, where $u_n = \inn{u,e_n}_{\H}$ we easily get \begin{equation}\label{proposition_identity1}
        \inn{Bu,u}_{\H} = \sum_{n \geq 0} \lambda_n u_n^2.
    \end{equation}
    Next, we bound $\inn{(A-B)u,u}_{\H}$ by doing:
    \begin{equation}
        \begin{split}
            \inn{(A-B)u,u}_{\H} &= \sum_{m,n}u_m u_n (\inn{Ae_m,e_n}_{\H}-\inn{Be_m,e_n}_{\H}) 
            \\
            &= \sum_{m,n}u_m u_n (\inn{Ae_m,e_n}_{\H}-\inn{B^{1/2}e_m,B^{1/2}e_n}_{\H})
            \\
            &= \sum_{m,n}u_m u_n (\inn{Ae_m,e_n}_{\H}- \inn{\lambda_m^{1/2} e_m,\lambda_n^{1/2} e_n}_{\H})
            \\
            &= \sum_{m,n} \lambda_m^{1/2}u_m \lambda_n^{1/2}u_n (\inn{A(\lambda_m^{-1/2} e_m),\lambda_n^{-1/2} e_n}_{\H} - \inn{e_m,e_n}_{\H})
            \\
            &= \sum_{m,n} \lambda_m^{1/2}u_m \lambda_n^{1/2}u_n (\inn{AB^{-1/2}e_m,B^{-1/2}e_n}_{\H} - \inn{e_m,e_n}_{\H})
            \\
            &= \sum_{m,n} \lambda_m^{1/2}u_m \lambda_n^{1/2}u_n \inn{(B^{-1/2}AB^{-1/2} - I) e_m,e_n}_{\H}
            \\
            &\leq \left(\sum_{m,n} \lambda_m u_m^2 \lambda_n u_n^2\right)^{1/2} \left( \sum_{m,n} \inn{(B^{-1/2}AB^{-1/2} - I) e_m,e_n}_{\H}^2 \right)^{1/2}
            \\
            &= \inn{Bu,u}_{\H} \norm{B^{-1/2}AB^{-1/2} - I}_{\L_2(\H)},
        \end{split}
    \end{equation}
where in the identities above we used in the third and fourth lines that $B^{\pm 1/2}u_n = \lambda^{\pm 1/2}u_n $, in the seventh line we used Cauchy-Schwarz for the double sum and in the last line we employed identity (\ref{proposition_identity1}). 
\end{proof} 
A direct implication of the result above is the following. 
\begin{corollary}\label{corollary_inequality}
    Let $Z^{(1)}$ and $Z^{(2)}$ be two {isotropic $\H$-valued} Gaussian random fields {defined on $\mathbb{S}^d$} and $\{ b_l^{(j)} \}_{l \geq 0}$, $j=1,2$ be their respective Schoenberg operator coefficient sequences. 
    Then:
    \begin{equation*}
        \sum_{l \geq 0} h(l) \left( \frac{\inn{b_l^{(2)}u,u}_{\H}}{\inn{b_l^{(1)}u,u}_{\H}} - 1 \right)^2 \leq \sum_{l \geq 0}h(l) \norm{(b_l^{(2)})^{-1/2}b_l^{(1)}(b_l^{(2)})^{-1/2} - I}_{\L_2(\H)}^2.
    \end{equation*}
\end{corollary}
\begin{proof}
    For each $l\geq 0$ we apply Lemma \ref{proposition_norms} to $A = b_l^{(2)}$ and $B = b_l^{(1)}$ to get $\left( \frac{\inn{b_l^{(2)}u,u}_{\H}}{\inn{b_l^{(1)}u,u}_{\H}} - 1 \right) \leq \norm{(b_l^{(2)})^{-1/2}b_l^{(1)}(b_l^{(2)})^{-1/2} - I}_{\L_2(\H)}$ and then perform the sum.
\end{proof}
} 

{Some interesting facts can be noted by considering the finite dimensional case.} 
Here we highlight as a particular case of the previous result when we take the Hilbert space to be a finite dimensional Euclidean space: $\H = \R^d$. In this case, the coefficients $b_l$ are positive definite $\R^{d \times d}$ matrices. The Hilbert-Schmidt norm $\L_2(\H)$ becomes the familiar Frobenius norm, but since all norms are equivalent in a finite-dimensional space, we can take any of them in the expression \ref{equiv_condition}.

\begin{corollary}
    For $j=1,2$, let $Z^{(j)}$ be isotropic $\R^p$-valued Gaussian random fields defined on $\S^d$ with distributions ${\bf P}_j$, whose covariance functions $\B_j$ admit decomposition \ref{B_schoenberg}. Then ${\bf P}_1 \equiv {\bf P}_2$ if and only if:
    \begin{equation}\label{equiv_condition_matrix}
        \sum_{l \geq 0}h(l) \norm{(b_l^{(2)})^{-1/2}b_l^{(1)}(b_l^{(2)})^{-1/2} - I}^2 < +\infty,
    \end{equation}
    where $\norm{\cdot}$ can be any of the equivalent matrix norms.
\end{corollary}

\subsection{Example: Multiquadratic Bivariate Model}
We illustrate the use of the main results in this work by introducing some bivariate families of random fields and finding the parameter conditions for equivalency of their distributions. The following findings are an extension of the work done for the scalar case in \cite{arafat2018equivalence}. We start by discussing the multiquadratic family, introduced in \cite{gneiting2013}.\\

Throughout, we abuse of notation ad we write $R(x,y)$, for $R: \S^d \times \S^d \to \R$, as $\psi(\cos \theta)$, for $\psi:[0,\pi] \to \R$ and $\theta$ being the arccosine of the inner product between $x$ and $y$.

\begin{proposition}[Multiquadratic bivariate covariance]\label{prop:multiquadratic}
    Let $\varphi_M(\theta|\alpha,\rho,\sigma)$ be the matrix valued function with entries: 
    \begin{equation}\label{multiquadratic}
        \varphi_{i,j,M}(\theta|\alpha,\rho,\sigma) = \rho_{i,j} \sigma_i \sigma_j  \frac{(1-\alpha_{i,j})^2}{\left ( 1+ \alpha_{i,j}^2 - 2\alpha_{i,j} \cos \theta \right )}   , \qquad i,j=1,2, \quad \theta \in [0,\pi],
    \end{equation}
where $\rho_{1,1} = \rho_{2,2} = 1$, $\rho_{1,2} = \rho_{2,1} \in (0,1).$ and $\alpha_{i,j} \in (0,1)$, $\alpha_{1,2} = \alpha_{2,1}$. If the parameters $(\alpha,\rho,\sigma)$ satisfy the conditions:
\begin{equation}\label{eq:corr_cond}
    \begin{split}
        &\alpha_{1,2} \leq \sqrt{\alpha_{1,1}\alpha_{2,2}},
        \\
        &\rho_{1,2} < \left( \frac{(1-\alpha_{1,1})(1-\alpha_{2,2})}{(1-\alpha_{1,2})^2} \right)^{(d-1)/2},
    \end{split}
\end{equation}
then the function $R(x,y) = \varphi_M(\theta|\alpha,\rho,\sigma)$ is a valid correlation function for an isotropic Gaussian random field $Z = (Z_1, Z_2)$ on $\S^d$.
\end{proposition}
\begin{proof}
    From \cite{arafat2018equivalence} we see that the Schoenberg coefficients of $R$ are the $2\times 2$ matrices $b_n$ with entries 
    \begin{equation}
        b_n(i,j) = \rho_{i,j}\sigma_i \sigma_j \binom{d+n-2}{n}\alpha_{i,j}^n (1-\alpha_{i,j})^{d-1}.
    \end{equation} 
In order to ensure that $R$ is a correlation functions we need each $\{b_n\}_{n\geq 0}$ to be positive definite, which happens if and only if $b_n(1,1) > 0$ and $\det(b_n) > 0$ for all $n \geq 0$. Indeed, $b_n(1,1) = \sigma_1^2\binom{d+n-2}{n}\alpha_{1,1}^n (1-\alpha_{1,1})^{d-1} > 0$. Furthermore:
\begin{equation*}
    \begin{split}
        &\det(b_n) > 0
        \\
        &\iff \binom{d+n-2}{n}\sigma_1^2 \sigma_2^2 \Bigg ( \alpha_{1,1}^n \alpha_{2,2}^n (1-\alpha_{1,1})^{d-1} (1-\alpha_{2,2})^{d-1} \\ & \qquad \qquad - \rho_{1,2}^2\alpha_{1,2}^{2n}  (1-\alpha_{1,2})^{2(d-1)} \Bigg) > 0
        \\
        &\iff \rho_{1,2} <\left( \frac{\alpha_{1,1}\alpha_{2,2}}{\alpha_{1,2}^2}\right)^{n/2} \left( \frac{(1-\alpha_{1,1})(1-\alpha_{2,2})}{(1-\alpha_{1,2})^2} \right)^{(d-1)/2}.
    \end{split}
\end{equation*}
Thus we must have $\rho_{1,2} < \inf_{n \geq 0} \left( \frac{\alpha_{1,1}\alpha_{2,2}}{\alpha_{1,2}^2}\right)^{n/2} \left(\frac{(1-\alpha_{1,1})(1-\alpha_{2,2})}{(1-\alpha_{1,2})^2} \right)^{(d-1)/2}$. Observe that if $\alpha_{1,2}^2 > \alpha_{1,1}\alpha_{2,2}$, then the infimum is zero and we get $\rho_{1,2} = 0$. This is the trivial case where $R$ is diagonal and the corresponding field $Z$ has independent components $Z_1$ and $Z_2$, which we are not considering here. Now, if we assume $\alpha_{1,2}^2 \leq \alpha_{1,1}\alpha_{2,2}$, then the expression inside the infimum is non-decreasing in $n$ and is minimized at $n=0$. Therefore $\rho_{1,2} < \left( \frac{(1-\alpha_{1,1})(1-\alpha_{2,2})}{(1-\alpha_{1,2})^2} \right)^{(d-1)/2}$ ensures that all $b_n$ are positive-definite, which completes the proof. 
\end{proof}

\begin{proposition}
Let $Z^{(j)}$, $j=1,2$, be two {$\mathbb{R}^2$-valued isotropic} Gaussian random fields {defined on $\mathbb{S}^d$} with respective multiquadratic correlations $\varphi(\theta|\alpha^{(j)},\rho^{(j)},\sigma^{(j)})$ as defined in Proposition \ref{prop:multiquadratic}, which satisfy \eqref{eq:corr_cond}. Then, they have equivalent distributions if and only if one of the following holds:
\begin{itemize}
    \item $\sigma^{(1)} = \sigma^{(2)}$, $\alpha_{i,i}^{(1)} = \alpha_{i,i}^{(2)} =: \alpha_{i,i}$ and $\alpha^{(i)}_{1,2} < \sqrt{\alpha_{1,1}\,\alpha_{2,2}}$, for $i=1,2$.
    \item $(\alpha^{(1)}, \rho^{(1)},\sigma^{(1)}) = (\alpha^{(2)}, \rho^{(2)},\sigma^{(2)})$, and $\alpha^{(i)}_{1,2} = \sqrt{\alpha_{1,1}\,\alpha_{2,2}}$ for at least one $i\in\{1,2\}$.
\end{itemize}
\end{proposition}
\begin{proof} 
We see that if $Z^{(1)}, Z^{(2)}$ have equivalent distributions, then Corollary \ref{corollary_inequality} implies:
    \begin{equation*}
        \sum_{l \geq 0} h(l) \left( \frac{b_l^{(2)}(i,i)}{b_l^{(1)}(i,i)} - 1 \right)^2 < +\infty, \quad \text{ for $i=1,2$},
    \end{equation*}
    that is, the marginals are equivalent.
However, in \cite{arafat2018equivalence} it is shown that a necessary and sufficient condition for this sum to be finite is that $\sigma_i^{(1)} = \sigma_i^{(2)}$ and $\alpha_{i,i}^{(1)} = \alpha_{i,i}^{(2)}=:\alpha_{i,i}$, $i=1,2$.
Under these constraints, $b_l^{(2)}$ can be written in terms of $b_l^{(1)}$ and the difference between the off-diagonal terms, i.e.,
$$
b_l^{(2)} = b_l^{(1)} + (b_l^{(2)}(1,2) - b_l^{(1)}(1,2)) \overline{I}, \qquad \text{with } \overline{I}= \begin{bmatrix}
0 & 1\\
1& 0
\end{bmatrix}.
$$
Hence, we can write
\begin{align*}
    \norm{(b_l^{(1)})^{-1/2}b_l^{(2)}(b_l^{(1)})^{-1/2} - I}^2 &= (b_l^{(2)}(1,2) - b_l^{(1)}(1,2))^2\norm{(b_l^{(1)})^{-1/2}\overline{I}(b_l^{(1)})^{-1/2}}^2.
\end{align*}
The matrix $(b_l^{(1)})^{-1/2}\overline{I}(b_l^{(1)})^{-1/2}$ is symmetric and, if we consider the Frobenius norm, we can write
\begin{align*}
\norm{(b_l^{(1)})^{-1/2}\overline{I}(b_l^{(1)})^{-1/2}}^2 &= \operatorname{trace}\left ( (b_l^{(1)})^{-1/2}\overline{I}(b_l^{(1)})^{-1/2} (b_l^{(1)})^{-1/2}\overline{I}(b_l^{(1)})^{-1/2}\right)\\
&= \operatorname{trace}\left ( (b_l^{(1)})^{-1}\overline{I} (b_l^{(1)})^{-1}\overline{I}\right)\\
&= 2 \frac{b_l^{(1)}(1,1)b_l^{(1)}(2,2) + (b_l^{(1)}(1,2))^2  }{(b_l^{(1)}(1,1)b_l^{(1)}(2,2) - (b_l^{(1)}(1,2))^2)^2} ,
\end{align*}
since
$$
(b_l^{(1)})^{-1}\overline{I} (b_l^{(1)})^{-1}\overline{I} = \frac{1}{\left(\operatorname{det}(b_l^{(1)})\right)^2}\begin{bmatrix}
    - b_l^{(1)}(1,2) & b_l^{(1)}(2,2)\\ 
b_l^{(1)}(1,1)& -b_l^{(1)}(1,2)
\end{bmatrix}
\begin{bmatrix}
    - b_l^{(1)}(1,2) & b_l^{(1)}(2,2)\\ 
b_l^{(1)}(1,1)& -b_l^{(1)}(1,2)
\end{bmatrix},
$$
and $\operatorname{det}(b_l^{(1)})=b_l^{(1)}(1,1)b_l^{(1)}(2,2) - (b_l^{(1)}(1,2))^2>0$. Thus,
\begin{align*}
    \norm{(b_l^{(1)})^{-1/2}b_l^{(2)}(b_l^{(1)})^{-1/2} - I}^2
    &=2 (\xi_l^{(2)} - \xi^{(1)}_l )^2 \frac{ 1 + (\xi^{(1)}_l)^2 }{(1 - (\xi^{(1)}_l)^2)^2} ,
\end{align*}
with $$\xi^{(j)}_l := \frac{b_l^{(j)}(1,2)}{\sqrt{b_l^{(j)}(1,1)b_l^{(j)}(2,2)}}, \qquad j=1,2.$$ Note that each $\xi^{(j)}_l$ is either a strictly positive constant in $l$ or tends to $0$, depending on whether $\alpha^{(j)}_{1,2} = \sqrt{\alpha_{1,1}\,\alpha_{2,2}}$ or $\alpha^{(j)}_{1,2} < \sqrt{\alpha_{1,1}\,\alpha_{2,2}}$, respectively.
As a consequence, equivalence holds, i.e.,
$$
\sum_{l=0}^\infty h(l)\norm{(b_l^{(1)})^{-1/2}b_l^{(2)}(b_l^{(1)})^{-1/2} - I}^2 < +\infty,
$$
if and only if $\sigma_i^{(1)} = \sigma_i^{(2)}$ and $\alpha_{i,i}^{(1)} = \alpha_{i,i}^{(2)} =:\alpha_{i,i}$, $i=1,2$, and
$$
\sum_{l=0}^\infty h(l) (\xi_l^{(2)} -\xi_l^{(1)})^2 = \sum_{l=0}^\infty h(l) \left (\frac{\xi_l^{(2)}}{\xi_l^{(1)}}-1 \right)^2 (\xi_l^{(1)})^2 < +\infty.
$$
If $\alpha^{(1)}_{1,2} = \alpha^{(2)}_{1,2}$ and $\rho^{(1)}_{1,2} =\rho^{(2)}_{1,2}$, clearly the series converges. Moreover, we can exclude the cases in which $\alpha^{(j)}_{1,2} = \sqrt{\alpha_{1,1}\,\alpha_{2,2}}$ for at least one $j\in\{1,2\}$, since it is readily seen that $(\xi_l^{(2)} -\xi_l^{(1)})^2$ dos not converge to zero, unless $\alpha^{(1)}_{1,2} = \alpha^{(2)}_{1,2}$ and $\rho^{(1)}_{1,2} =\rho^{(2)}_{1,2}$.

Now assume \(\alpha^{(1)}_{1,2} < \sqrt{\alpha_{1,1}\,\alpha_{2,2}}\) and \(\alpha^{(2)}_{1,2} < \sqrt{\alpha_{1,1}\,\alpha_{2,2}}\), and consider $\alpha^{(1)}_{1,2} \ne \alpha^{(2)}_{1,2}$, regardless of $\rho^{(1)}_{1,2}$ and $\rho^{(2)}_{1,2}$.
To study the behavior of the series, we use Raabe's test which reduces to study the following limit
\begin{align*}\label{raabe}
 L:= \lim_{l \to +\infty}l \left (\frac{h(l)}{h(l+1)} \frac{\left( \frac{\rho^{(2)}_{1,2}}{\rho^{(1)}_{1,2}} \left(\frac{\alpha^{(2)}_{1,2}}{\alpha^{(1)}_{1,2}}  \right)^{l} \left (\frac{1-\alpha^{(2)}_{1,2}}{1-\alpha^{(1)}_{1,2}}  \right)^{d-1} -1\right)^2}{\left( \frac{\rho^{(2)}_{1,2}}{\rho^{(1)}_{1,2}} \left( \frac{\alpha^{(2)}_{1,2}}{\alpha^{(1)}_{1,2}}  \right)^{l+1} \left (\frac{1-\alpha^{(2)}_{1,2}}{1-\alpha^{(1)}_{1,2}}  \right)^{d-1} -1 \right)^2 }\left(\frac{\sqrt{\alpha_{1,1}\alpha_{2,2}}}{\alpha_{1,2}^{(1)}}\right)^2  -1 \right).
\end{align*}
Observe that
$$
\frac{h(l)}{h(l+1)} = \frac{(2l + d - 1)(l + 1)}{(2l + d + 1)(l + d - 1)} = 1 - \frac{d-1}{l} + O(l^{-2}),\qquad l\to\infty,
$$
and 
\begin{eqnarray*}
    && \lim_{l\to \infty} \frac{\left( \frac{\rho^{(2)}_{1,2}}{\rho^{(1)}_{1,2}} \left(\frac{\alpha^{(2)}_{1,2}}{\alpha^{(1)}_{1,2}}  \right)^{l} \left (\frac{1-\alpha^{(2)}_{1,2}}{1-\alpha^{(1)}_{1,2}}  \right)^{d-1} -1\right)^2}{\left( \frac{\rho^{(2)}_{1,2}}{\rho^{(1)}_{1,2}} \left( \frac{\alpha^{(2)}_{1,2}}{\alpha^{(1)}_{1,2}}  \right)^{l+1} \left (\frac{1-\alpha^{(2)}_{1,2}}{1-\alpha^{(1)}_{1,2}}  \right)^{d-1} -1 \right)^2 }\left(\frac{\sqrt{\alpha_{1,1}\alpha_{2,2}}}{\alpha_{1,2}^{(1)}}\right)^2 \\ &=& 
\begin{cases}
    \left(\frac{\sqrt{\alpha_{1,1}\alpha_{2,2}}}{\alpha_{1,2}^{(1)}}\right)^2 & \text{if } \alpha^{(2)}_{1,2} < \alpha^{(1)}_{1,2}\\
       \left(\frac{\sqrt{\alpha_{1,1}\alpha_{2,2}}}{\alpha_{1,2}^{(2)}}\right)^2 & \text{if } \alpha^{(1)}_{1,2} < \alpha^{(2)}_{1,2}
\end{cases}.
\end{eqnarray*}

Hence, $L = +\infty$ and the series converges.
Similarly it holds for the case $\alpha^{(1)}_{1,2} = \alpha^{(2)}_{1,2}$ and $\rho^{(1)}_{1,2} \ne \rho^{(2)}_{1,2}$, and the proof is concluded.
\end{proof}

\subsection{Example: Legendre-Matérn Model}
We consider now the use of the results on a infinite-dimensional example. For this purpose,
let $L^2([0,1])$ be the space of square integrable periodic functions on $[0,1]$. We choose then the $Z^{(j)}$'s to be isotropic $L^2([0,1])$-valued random fields, i.e., $\H = L^2([0,1])$.

We additionally model the coefficients $a^{(j)}_{l,m}$'s as real stationary processes on $L^2([0,1])$, in order to obtain the following expansion for the corresponding covariance operators
$$
b^{(j)}_l   = \sum_{k = -\infty}^\infty \gamma^{(j)}_{l,k} \, e^{i2\pi k\cdot} \otimes  e^{i 2\pi k \cdot}.
$$
The $\gamma^{(j)}_{l,k}$'s are real numbers such that $\gamma^{(j)}_{l,k} = \gamma^{(j)}_{l,-k}$,
$$
\gamma^{(j)}_{l,k} > 0, \qquad  \sum_{l\ge 0} h(l) \sum_{k=-\infty}^\infty \gamma^{(j)}_{l,k} =  \sum_{l\ge 0} h(l) \left( \gamma^{(j)}_{l,0} +2 \sum_{k=1}^\infty \gamma^{(j)}_{l,k}\right) < \infty .
$$
Then, the equivalence condition is given by
$$
\sum_{l \ge 0 } \sum_{k \ge 0} h(l) \left ( \frac{\gamma^{(2)}_{l,k}}{\gamma^{(1)}_{l,k}} -1  \right)^2 < +\infty.
$$
We restrict our attention to the 2-dimensional sphere $\mathbb{S}^2$ and consider an extension of the Legendre-Matérn covariance function (see \cite{GuinnessFuentes2016}) by specifying
\begin{equation}\label{legendre-matérn}
    h(l)\gamma_{l,k} = \frac{\sigma^2}{(\alpha + k^2 + l^{2})^{\nu + 1/2}}
\end{equation}
with $\alpha,\nu, \sigma > 0$. 
\begin{proposition}
Let $Z^{(j)}$, $j=1,2$, be two isotropic $ L^2([0,1])$-valued Gaussian random fields on $\mathbb{S}^2$ with Legendre-Matérn coefficients \eqref{legendre-matérn} of parameters $\alpha^{(j)}, \nu^{(j)},\sigma^{(j)}>0$, $j=1,2$. Then, they have equivalent distributions if and only if $\sigma^{(1)} = \sigma^{(2)}$ and $\nu^{(1)} = \nu^{(2)}$.
\end{proposition}

\begin{proof}
The equivalence condition for $Z^{(1)}$ and $Z^{(2)}$ under the Legendre-Matérn model is given by
$$
 \sum_{l \ge 0 }  \sum_{k \ge 0 } h(l) \left( \left(\frac{\sigma^{(2)}}{\sigma^{(1)}} \right)^2\frac{(\alpha^{(1)} + k^2 + l^{2})^{\nu^{(1)} + 1/2}}{(\alpha^{(2)} + k^2 +l^{2})^{\nu^{(2)} + 1/2}} -1\right)^2 < \infty.
$$
Clearly, if $\sigma^{(1)}=\sigma^{(2)}, \ \alpha^{(1)}=\alpha^{(2)}, \ \nu^{(1)}=\nu^{(2)}$, the series converges. When $\nu^{(1)}\ne\nu^{(2)}$, the series diverges: indeed, by considering for instance the case $k=0$, we observe that
$$
\left ( \frac{\gamma^{(2)}_{l,0}}{\gamma^{(1)}_{l,0}} -1  \right)^2 = \left( \left(\frac{\sigma^{(2)}}{\sigma^{(1)}} \right)^2\frac{(\alpha^{(1)} + l^{2})^{\nu^{(1)} + 1/2}}{(\alpha^{(2)} + l^{2})^{\nu^{(2)} + 1/2}} -1\right)^2
$$
does not converge to 0 as $l\to \infty$. Similarly, if $\nu^{(1)}=\nu^{(2)}$ and $\sigma^{(1)}\ne\sigma^{(2)}$.

As a consequence, we restrict our attention to the the case $\sigma^{(1)}=\sigma^{(2)}$ and $ \nu^{(1)}=\nu^{(2)}=:\nu$, with $\alpha^{(1)}\ne\alpha^{(2)},$ which gives
$$
\frac{\gamma^{(2)}_{l,k}}{\gamma^{(1)}_{l,k}} = \left (\frac{\alpha^{(1)}+k^2 + l^{2}}{\alpha^{(2)}+k^2 + l^{2}}\right)^{\nu + 1/2} =  \left (1+\frac{\alpha^{(1)} - \alpha^{(2)}}{\alpha^{(2)} +k^2 + l^{2}}\right)^{\nu + 1/2}.
$$
For $l$ and $k$ large enough, we can use the Binomial/Taylor expansion and get
\begin{align*}
    \left (1+\frac{\alpha^{(1)} - \alpha^{(2)}}{\alpha^{(2)} +k^2+ l^{2}}\right)^{\nu + 1/2} = 1 + (\nu+1/2) \frac{\alpha^{(1)} - \alpha^{(2)}}{\alpha^{(2)} + k^2+ l^{2}} + O((k^2+l^2)^{-2}).
\end{align*}
Hence, we have 
\begin{align*}
 \left( \left (\frac{\alpha^{(1)}+k^2 + l^{2}}{\alpha^{(2)}+k^2 + l^{2}}\right)^{\nu + 1/2} -1 \right)^2 &= (\nu+1/2)^2 \frac{(\alpha^{(1)} - \alpha^{(2)})^2}{(k^2 + l^{2})^2} + O((k^2+l^2)^{-3}) \\&= O((k^4+l^4)^{-1}).
\end{align*}
Since $h(l)=2l+1$, the double series is finite and the proof is concluded.

\end{proof}

\section{Conclusions} \label{sec4}
This paper develops a comprehensive framework for the study of functional Gaussian random fields on hyperspheres, extending classical results on scalar and vector-valued fields to the infinite-dimensional setting. By introducing an operator-valued analogue of the Schoenberg representation, we have characterized the covariance structure of Hilbert-valued isotropic spherical fields in terms of trace-class 
$d$-Schoenberg sequences. The provided representation sets a 
 natural harmonic-analytic decomposition of the field and 
provides the background for studying the probabilistic properties in terms of equivalence of Gaussian measures.  \par 
The main theoretical contribution of this paper is the
  functional version of the Feldman–Hájek criterion for Gaussian measures on functional spaces. The resulting characterization establishes that equivalence of Gaussian measures is governed by a square-integrability condition involving the operator-valued Schoenberg coefficients. This condition is the exact infinite-dimensional analogue of previously known results for scalar fields and yields a precise spectral description of when two functional spherical models generate mutually absolutely continuous laws. \par
  Our findings have then been supported through marginalization, for which we show that  the functional criterion controls the equivalence of all scalar projections of the field. This provides important intuition: although the functional field lives in an infinite-dimensional space, its measure-theoretic behavior is encoded in a structured family of scalar Schoenberg coefficients obtained through directional evaluations. The inequalities we derive demonstrate that the functional criterion dominates its scalar counterparts, ensuring consistency across all one-dimensional projections.
The examples we present illustrate the practical consequences of the theory. For multiquadratic bivariate models, the equivalence conditions reduce to sharp constraints on the parameters governing cross-correlation and angular dependence. For the infinite-dimensional Legendre–Mat{\'e}rn family, the spectral characterization leads to a clear description of the role of smoothness and scale parameters in determining equivalence. These examples show that the abstract theory yields concrete and interpretable results for widely used covariance families.
Beyond the theoretical contribution, the results have implications for several application domains. In spatial and spatio-temporal statistics, the equivalence framework is central to understanding asymptotics under infill sampling and the robustness of kriging under covariance misspecification. In functional data analysis, these results provide a principled foundation for modeling Hilbert-valued spherical data such as climate model fields, remote sensing images, neural activity patterns on spherical domains, and manifold-valued machine learning kernels. More broadly, the operator-valued Schoenberg framework opens new research avenues for the study of Gaussian processes on non-Euclidean and high-dimensional structures, particularly in the context of reproducing kernel Hilbert spaces and operator-valued kernels on manifolds.
Future work may explore extension to non-Gaussian fields, the characterization of equivalence under space–time or product manifolds, connections with Bayesian inverse problems on spheres, and the design of statistical procedures that explicitly exploit the operator-valued structure described here. The theory developed in this paper provides a rigorous mathematical foundation for such developments and, we hope, will stimulate further research at the interface of probability, functional analysis, and spatial statistics.

\bibliographystyle{apalike}
\bibliography{refs}
\end{document}